\documentclass[12pt]{article}%
\usepackage{graphicx}
\usepackage{amsmath}
\usepackage{amsfonts}
\usepackage{amssymb}%
\providecommand{\U}[1]{\protect\rule{.1in}{.1in}}
\newtheorem{theorem}{Theorem}

\newtheorem{lemma}[theorem]{Lemma}

\newenvironment{proof}[1][Proof]{\noindent\textbf{#1.} }{\ \rule{0.5em}{0.5em}}
\begin{document}

\title{Poisson Hypothesis for Open Networks at Low Load\footnotetext{\noindent Part
of this work has been carried out in the framework of the Labex Archimede
(ANR-11-LABX-0033) and of the A*MIDEX project (ANR-11-IDEX-0001-02), funded by
the \textquotedblleft Investissements d'Avenir" French Government programme
managed by the French National Research Agency (ANR). Part of this work has
been carried out at IITP RAS. The support of Russian Foundation for Sciences
(project No. 14-50-00150) is gratefully acknowledged.}}
\author{A. Rybko$^{\dag\dag}$, Senya Shlosman$^{\dag,\dag\dag}$ and A.
Vladimirov$^{\dag\dag}$\\$^{\dag}$Aix Marseille Universit\'{e}, Universit\'{e} de Toulon, \\CNRS, CPT UMR 7332, 13288, Marseille, France\\senya.shlosman@univ-amu.fr \\$^{\dag\dag}$Inst. of the Information Transmission Problems,\\RAS, Moscow, Russia \\rybko@iitp.ru, shlos@iitp.ru, vladim@iitp.ru}
\maketitle

\begin{abstract}
We study large communication networks of the mean-field type. The input flows
to the nodes of the network are supposed to be stationary and with low rate.
We show that such a network is ergodic, i.e. it goes to the stationary state,
which does not depend on the initial state of the network. This is in contrast
with the high load regime, when the large time behavior of the network might
depend on its initial state. Our technique is based on the coupling
construction, which couples two Non-Linear Markov Processes.

\end{abstract}

\section{Introduction and results}

In this paper we study large information networks with low load. Our goal is
to show that they obey a certain pattern, which is known under the name of
Poisson Hypothesis.

The Poisson Hypothesis is a device to predict the behavior of large queuing
networks. It was formulated first by L. Kleinrock in \cite{K}, and it was
assumed to hold universally. It concerns the following situation. Suppose we
have a large network of servers, through which many customers are traveling,
being served at different nodes of the network. If the node is busy, the
customers wait in the queue. Customers are entering into the network from the
outside via some nodes, and these external flows of customers are Poissonian,
with constant rates. The service time at each node is random, depending on the
node, and the customer. The PH prediction about the (long-time, large-size)
behavior of the network is the following:

\begin{itemize}
\item consider the total flow $\mathcal{F}$ of customers to a given node
$\mathcal{N}.$ Then $\mathcal{F}$ is approximately equal to a Poisson flow,
$\mathcal{P},$ with a time dependent rate function $\lambda_{\mathcal{N}%
}\left(  T\right)  .$

\item The exit flow from $\mathcal{N}$ -- \textbf{not} Poissonian in general!
-- has a rate function $\gamma_{\mathcal{N}}\left(  T\right)  ,$ which is
smoother than $\lambda_{\mathcal{N}}\left(  T\right)  $ (due to averaging,
taking place at the node $\mathcal{N}$).

\item As a result, the flows $\lambda_{\mathcal{N}}\left(  T\right)  $ at
various nodes $\mathcal{N}$ should go to a constant limits $\bar{\lambda
}_{\mathcal{N}}\approx\frac{1}{T}\int_{0}^{T}\lambda\left(  t\right)  dt$, as
$T\rightarrow\infty,$ the flows to different nodes being almost independent.

\item The above convergence is uniform in the size of the network.
\end{itemize}

Note that the distributions of the service times at the nodes of the network
can be arbitrary, so PH deals with quite a general situation. The range of
validity of PH is supposed to be the class of networks where the internal flow
to every node $\mathcal{N}$ is a union of flows from many other nodes, and
each of these flows constitute only a small fraction of the total flow to
$\mathcal{N}.$ If true, PH provides one with means to make easy computations
of quantities of importance in network design.

The rationale behind this conjectured behavior is natural: since the inflow is
a sum of many small inputs, it is approximately Poissonian. And due to the
randomness of the service time the outflow from each node should be
\textquotedblleft smoother\textquotedblright\ than the total inflow to it.
(This statement was proven in \cite{RSV1} under quite general conditions.) In
particular, the variation of the latter should be smaller than that of the
former, and so all the flows should go with time to corresponding constant values.

In the paper \cite{RS1} the Poisson Hypothesis is proven for simple networks
in the infinite volume limit, under some natural conditions. The
counterexamples to such universal behavior are also known by now. Some were
constructed in \cite{RS2}, \cite{RSV2}.\ \ 

It was argued in \cite{RSV2}, that the violation of the Poisson Hypothesis can
happen only at high load, and that the load plays here the same role as the
temperature in the statistical mechanics, so that the high (low) load
corresponds to low (high) temperature. We have shown in \cite{RSV3} that
indeed the Poisson Hypothesis holds for the closed networks with low load. In
the present paper we show that the Poisson Hypothesis behavior holds also for
the open networks with low load. (The time gap between the paper \cite{RSV3}
and the present one is due to the fact that certain general results needed
were not available earlier. Now they are established, see \cite{BRS}.)

The main technical innovation of our paper is that we prove our result by
constructing a suitable coupling of two Non-Linear Markov Processes (in
Sections \ref{cou1}, \ref{cou2} below). The coupling constructed shows that
the two NLMP merge together as time diverges, provided the load of the network
is low. It is thus reminiscent of the FKG-uniqueness results of statistical
mechanics, when it is obtained via coupling constructions for systems with
attraction, \cite{H}.

In the next Section we describe the basic finite network, from which large
networks of the mean-field type can be built. This mean-field construction is
explained in the Section 3, and the corresponding Non-Linear Markov Process
(NLMP) is defined in Section 4. The general properties of the networks with
low input flows are established in the Section 5. The Section 6 contains the
main result of the paper -- that of ergodicity of the NLMP, corresponding to
the low inflow. It is established by using the coupling of Sections
\ref{cou1}, \ref{cou2}.

\section{Basic network}

The main object of our study are large queueing networks that are composed of
many identical copies of a finite basic network $G$. In this section we will
describe this basic network.

\subsection{$G:$ Markov chain of queues}

The network $G$ consists of a finite number of servers, $s,$ also called
nodes. At any given time $t$, each node is populated by a finite number of
customers (also called particles) of different classes. The set $J$ of all
classes is finite, $|J|<\infty$. It is convenient from the very beginning to
include into the class $j$ of the customer the address of server $s\left(
j\right)  \in G$ where the customer is currently located. The customers wait
in the queue for their turn to be served by the node. The order of service is
governed by the service discipline. After being served, the customer either
jumps to some node of $G$ (becoming there the last in the queue), or leaves
the network.

The traffic of customers in $G$ is governed by the routing matrix $P$. This is
a $|J|\times|J|$ matrix, whose rows and columns correspond to the classes of
customers in the network. The element $p_{ij}$ of $P$ is the probability of a
customer $c$ of class $i$ to change its class from $i$ to $j$ after the
service. As soon as the service of $c$ is over, the customer jumps with
probability $p_{ij}$ to the node $s\left(  j\right)  .$ With the complementary
probability $p_{i0}=1-\sum_{j\in J}p_{ij}$ the customer leaves the network.

\subsection{Queues}

We restrict consideration to Markov networks where the service time of each
customer is exponentially distributed. We also consider service disciplines
that depend only on the order of arrival of customers to the queue (and not on
the interarrival times or any other history). Namely, for each node $n$ and
each queue $x=\left(  j_{1},...,j_{k}\right)  $ of length $k$ at this node, a
collection $\gamma_{n}(x,j_{r}):r=1,\dots,k$ of service rates is defined by
the service discipline. The formal definition will be given later.

For the alphabet $J=\{1,\dots,|J|\}$ we define the configuration space
\[
X=J^{\ast}=\bigcup_{k=0,1,\dots}J^{k}%
\]
of finite words in this alphabet, with $J^{0}=\{\varnothing\}$ and $J^{1}=J$.
We will use the notation $x=(x_{1},\dots,x_{k})\equiv(j_{1},\dots,j_{k})$ for
the list of all customers in the network. Denote $|x|=k$.

The order of customers in the queue coincides with the order of their arrival,
that is, $x_{1}$ is the class of the oldest customer in the queue and $x_{k}$
is that of the last arrived. (The customer $x_{1}$ is not necessarily the
first to be served, since we consider a rather general class of service disciplines.)

\subsection{Network events}

The current state $x\in X$ can be changed as a result of arrival events and
service events. An \emph{arrival event} of type $j$ takes $x$ to
\[
x\oplus j=\{j_{1},\dots,j_{|x|},j\}.
\]
\emph{Service event} may only happen if $x\neq\emptyset$. If $x=\{j_{1}%
,\dots,j_{k-1},j_{k},j_{k+1},\dots,j_{|x|}\},$ i.e. $k\leq\left\vert
x\right\vert ,$ then the service event operation $\ominus\ast_{k}$ is the jump
from $x$ to
\[
x\ominus\ast_{k}=\{j_{1},\dots,j_{k-1},j_{k+1},\dots,j_{|x|}\}.
\]
With some abuse of notation we also denote this jump as $x\ominus j_{k}.$
Service and arrival events may happen simultaneously, if a particle whose
service is finished, joins the same queue again, see below.

\subsection{Markov dynamics}

The random evolution of the network $G$ is a continuous-time Markov process
$\mathcal{G}$ on the countable configuration space $X$. Let us write down the
rates of all arrival and service events as functions of the current
configuration $x\in X$.

The external arrivals of clients of type $j\in J$ is a Poisson process with
rate $\lambda_{j}(t)$. We will consider the time homogeneous case, when
$\lambda_{j}(t)\equiv\lambda_{j}$, though for technical reasons we will also
use sometime the general case of time-dependent rates $\lambda_{j}(t).$ The
external arrival events of customers of different types $j\in J$ are
independent of each other and of the state of the system.

A service event of type $i\in x$ happens at appropriate location in $x$ at
rate $\gamma(x,i)$. If the customer $i$ stays in the network and turns into
customer $j,$ then the configuration of the system becomes $x\ominus i\oplus
j.$ That can happen with probability $p_{ij}.$ Another option is that the
customer leaves the network, and the network state becomes $x\ominus i.$ That
happens with probability $p_{i0}=1-\sum_{j\in J}p_{ij}.$

Apart from the pairs of associated service and arrival events mentioned above,
all the events in the system happen independently of each other. The non-zero
transition rates $r(x,y)$ of the process $G$ are thus equal to
\begin{align}
r(x,x\oplus j)  &  =\lambda_{j}(t), &  &  j\in J,\\
r(x,x\ominus i)  &  =p_{i0}\gamma(x,i), &  &  i\in x,\\
r(x,x\ominus i\oplus j)  &  =p_{ij}\gamma(x,i), & i\in x,  &  \ j\in J.
\end{align}

The state of the process $\mathcal{G}$ at the time $t$ is given by a
probability measure $\mu_{t}$ on $X$. We denote by
\[
\mathcal{P}(X)=\{\mu\in{\mathbb{R}}^{X}:\mu\left(  x\right)  \geq0,\ \ x\in
X,\ \ \sum_{x\in X}\mu\left(  x\right)  =1\}
\]
the space of probability measures on $X.$

\section{Mean-field networks}

In order to construct the mean-field networks $G_{M}$, we take the colection
of $M$ independent copies of $G$, and interconnect them as follows. The
service event $i$ at any given copy of $G$ is followed by the corresponding
arrival event $j$ at any other copy of $G$ with probability $\frac{1}{M}%
p_{ij}$. With the complementary probability, $p_{i0}=1-\sum_{j\in J}p_{j_{i}%
j}$, no arrival event happens, and the customer leaves the system.

\subsection{$\mathcal{G}_{M}:$ the mean-field type process}

The resulting Markov process $\mathcal{G}_{M}$ on the configuration product
space
\[
X^{M}=\{\mathbf{x}=(x_{1},\dots,x_{M}):x_{m}\in X,\ m=1,\dots,M\}
\]
can be described formally as follows. Let us introduce the notations:
\[
\mathbf{x}\oplus_{m}j=(x_{1},\dots,x_{m}\oplus j,\dots,x_{M})
\]
and for $i\in x_{m}$
\[
\mathbf{x}\ominus_{m}i=(x_{1},\dots,x_{m}\ominus i,\dots,x_{M}).
\]
Then the transition rates of the process $\mathcal{G}_{M}$ are given by the
relations
\begin{align}
r(\mathbf{x},\mathbf{x}\oplus_{m}j)  &  =\lambda_{j}(t),\\
r(\mathbf{x},\mathbf{x}\ominus_{m}i)  &  =(1-\sum_{j\in J}p_{ij})\gamma
(x_{m},i),\\
r(\mathbf{x},\mathbf{x}\ominus_{m}i\oplus_{m^{\prime}}j)  &  =\frac
{p_{ij}\gamma(x_{m},i)}{M},
\end{align}
for all $\quad m,m^{\prime}=1,\dots,M$, $i\in x_{m}$, $j\in J$.

\subsection{Factor-process}

The permutation group $S_{M}$ acts on $X^{M},$ and the process $\mathcal{G}%
_{M}$ is invariant under this action. The factor space can be identified with
the atomic measures on $X$. So we will use the notation
\begin{equation}
\mathbf{x}=\frac{1}{M}\sum_{m=1}^{M}\delta_{x_{m}}, \label{QE1}%
\end{equation}
where $\delta_{x}$ is the unit measure at $x\in X$.

Denote by $\mathcal{P}_{M}(X)\subseteq\mathcal{P}(X)$ the subspace of all
measures of the form (\ref{QE1}). The factor-process, denoted by the same
symbol, $\mathcal{G}_{M},$ becomes a continuous-time Markov process on the
(countable) configuration space $\mathcal{P}_{M}(X)$. Its state at time $t$ is
a probability measure%

\[
\mu_{t}^{M}=\{\mu_{t}^{M}(\mathbf{x}):\mathbf{x}\in\mathcal{P}_{M}(X)\}
\]
on $\mathcal{P}_{M}(X)$ and thus on $\mathcal{P}(X)$. The expectation
${\mathbb{E}}\mu_{t}^{M}$ is an element of $\mathcal{P}(X)$ but, in general,
not an element of $\mathcal{P}_{M}(X)$.

\section{Non-linear Markov process}

Let us consider the limit of $\mathcal{G}_{M}$ as $M\rightarrow\infty$. To
this end, let us choose one of $M$ copies of $G$ in the system $G_{M}$ -- say,
the first one, $G_{1}$ -- and observe for $t\geq0$ the arrival and service
events only at this copy. In other words, we take a projection $\mathcal{H}%
_{M}$ of $\mathcal{G}_{M}$ to $G_{1},$ integrating out the remaining degrees
of freedom.

If we want to know the evolution of the process $\mathcal{H}_{M},$ then, in
addition to the initial configuration $x_{0}\in X$ we need, of course, to
specify the initial state $\mu_{t=0}^{M}\in\mathcal{P}_{M}(X)$ of the whole ensemble.

Suppose that the initial measures $\mu_{0}^{M}$ converge weakly to some
$\mu_{0}\in\mathcal{P}(X)$ as $M\rightarrow\infty,$ and consider the limiting
process $\mathcal{H=}\lim_{M\rightarrow\infty}\mathcal{H}_{M}.$ The existence
and uniqueness of the limiting process is proven in \cite{BacQu13}. This
stochastic process will be referred to as the \emph{non-linear Markov process}
(NLMP) $\mathcal{H}$. To specify its evolution we need to know the initial
configuration $x\in X$ and the initial measure $\mu_{0}\in\mathcal{P}(X)$. The
NLMP $\mathcal{H}$ is not a Markov process on $X$ because of its dependence on
the whole measure $\mu_{0}$ and not just on $x(0)$.

The current configuration of this extended process is the pair $(\mu
_{t},x(t))$, and this configuration defines the distribution of the future
configurations $(\mu_{s},x(s))$, $s\geq t$, uniquely. Moreover, the measure
$\mu_{s}$ is defined by $\mu_{t}$ in a deterministic way, for $s\geq t$. The
measure $\mu_{t}$ is a solution of a system of differential equations which we
write down below. Once the family $\left\{  \mu_{t},t\geq0\right\}  $ is
given, the process $x(t)$ becomes an inhomogeneous Markov process on $X$, see
below for more details.

\subsection{The process equations}

The evolution equations for the state $\mu_{t}$ of the process $\mathcal{H}$
are conveniently written in terms of \emph{ flow rates} which we define now.

The \emph{ outflow rate} of class $j$ customers from the system $\mathcal{H}$
at time $t$ is given by
\begin{equation}
u_{j}(t)=\sum_{x\in X}\mu_{t}(x)\sum_{r:x_{r}=j}\gamma(x,x_{r}),\qquad j\in J
\label{YE2a}%
\end{equation}

The \emph{ inflow rate} of class $j$ customers to the system is then equal to
\begin{equation}
v_{j}(t)=\lambda_{j}(t)+w_{j}(t),\qquad j\in J, \label{YE3}%
\end{equation}
where
\begin{equation}
w_{j}(t)=\sum_{i\in J}p_{ij}u_{i}(t),\qquad j\in J, \label{YE3a}%
\end{equation}
is the \textit{internal inflow rate}.

Thus the total inflow is a time-inhomogeneous (vector) Poisson process, whose
rate vector is $v(t)=\lambda(t)+w(t)$. The queueing process $x(t)$ is then an
inhomogeneous Markov process corresponding to the time-dependent arrival
process with rates $v_{j}(t)$, $j\in J$, and the service rates $\gamma(x,i)$,
$i\in x$.

An arrival event of type $j$ converts the configuration from $x$ to $x\oplus
j$. The rate of this arrival equals to $\mu\left(  x\right)  v_{j}$.
Analogously, a service event of type $i$ converts $x$ to $x\ominus i$. Its
rate is $\mu\left(  x\right)  \gamma(x,i)$.

We have described the dynamics of $\mathcal{H}$ in terms of arrival and
service events by specifying their rates and the corresponding changes of the
measure $\mu$. The time evolution of $\mu_{t}\in\mathcal{P}(X)$ is given by
the following countable system of ordinary differential equations:
\begin{equation}
\dot{\mu}_{t}=\sum_{x\in X}\mu_{t}(x)\left[  \sum_{r=1}^{|x|}\gamma
(x,x_{r})(\delta_{x\ominus x_{r}}-\delta_{x})+\sum_{j\in J}v_{j}%
(t)(\delta_{x\oplus j}-\delta_{x})\right]  . \label{YE7}%
\end{equation}

Here we treat the measure $\mu_{t}$ as a vector with coordinates $\mu_{t}(x),$
and $\delta_{y}$-s denotes the $\delta$-measures at various points $y\in X.$
The first sum in brackets corresponds to service events that happen at
configurations $x\in X$, $|x|>0$, with (configuration-dependent) rates
$\gamma(x,j)$, and the second sum corresponds to arrival events that happen at
all configurations $x\in X$ with (configuration-independent) rates $v_{j}(t)$,
$j\in J$.

Relations (\ref{YE2a})--(\ref{YE7}) describe completely the evolution of
$\mu_{t}$. The resulting dynamical system is non-linear since its right-hand
side contains terms $\mu_{t}(x)v_{j}(t)$, which are quadratic in $\mu,$ since
$v_{j}(t)$ depends on $\mu_{t}$ linearly. The existence and uniqueness of the
solution $\mu_{t}$ on $[0,+\infty)$ of the system (\ref{YE7}) for any initial
state $\mu_{0}$ is proven in \cite{BRS} under minimal assumptions on the
network. In particular, the existence and uniqueness hold under the following
\emph{rate conditions}: for some $\lambda_{+}<\infty$ and $0<\gamma_{-}%
\leq\gamma_{+}<\infty$ we have
\begin{equation}
\lambda_{j}(t)\leq\lambda_{+}\quad\text{for all}\quad j\in J,\ t\geq0,
\label{RC1}%
\end{equation}%
\begin{equation}
\gamma_{-}\leq\gamma(x)\leq\gamma_{+}\quad\text{for all}\quad x\in
X,\ x\neq\emptyset, \label{RC2}%
\end{equation}
where $\gamma(x)=\sum_{r=1}^{|x|}\gamma(x,x_{r})$.

Note that under the rate conditions the total inflow rates $v_{j}(t)$ in the
system have a uniform upper bound $V<\infty$. So for each $x\in J$ the
function $\mu_{t}(x)$ is Lipschitz continuous in $t,$ with a Lipschitz
constant which does not depend on $x,t$, or $\mu_{0}$.

\section{Processes with low inflows}

\subsection{Bounds for a single queue}

Now let us consider a server that receives independent Poisson inflows of
particles of classes $j\in J$ with rates $\lambda_{j}(t)$. As in the NLMP
model, there is no feedback, that is, served customers never return to our server.

In what follows we assume that $\lambda_{j}(t)\leq\varepsilon$ for all
$t\geq0$, where $\varepsilon>0$ is a small parameter. Let us derive some
bounds on the length $|x|$ of the queue at this server. Recall that our
service discipline is rather general, but the total service rate $\gamma(x)$
is neither too low nor too high for any $x\neq0$, see rate conditions
(\ref{RC2}).

We will derive bounds on the distribution $\nu_{t}$ of queue length, provided
the queue is empty at $t=0$. Denote by $Q$ the random process of queue
lengths. The trajectories of $Q$ are step functions, taking values in
${\mathbb{Z}}_{+}=\{0,1,2,\dots\}$.

Denote by $\nu_{t}(k)$ the probability that there are exactly $k$ particles in
the queue at time $t$. We have
\[
\sum_{k=0,1,\dots}\nu_{t}(k)=1,\quad t\geq0.
\]

Let us compare the process $Q$ with the stationary process $Q^{\prime}$,
defined as follows:

\begin{itemize}
\item All the customers in $Q^{\prime}$ have the same exponentially
distributed service time and the FIFO service discipline is used;

\item The service rate in the process $Q^{\prime}$ is equal to the lower bound
${\gamma}_{-}$ on the service rates for the process $Q$;

\item The inflow to $Q^{\prime}$ is a stationary Poisson flow with rate
$\varepsilon$ for each type $j\in J.$
\end{itemize}

For $\varepsilon$ small enough such a process $Q^{\prime}$ exists and is
unique. One easily sees that the queues in $Q^{\prime}$ dominates
stochastically those in $Q$ at any time $t\geq0$. Indeed, by coupling we may
organize the pairs of evolutions $q(t)$ and $q^{\prime}(t)$ in such a way
that, apart from synchronous arrivals to both queues, there are additional
arrivals to $q^{\prime}$, and, apart from synchronous service attempts, there
are additional service attempts at $q$. Then $|q^{\prime}(t)|\geq|q(t)|$ for
all $t\geq0$ since $|q(0)|=0$.

For the stationary process $Q^{\prime}$ the explicit form of the queue length
distribution $\nu^{\prime}$ is known; it is the geometric distribution:
\[
\nu^{\prime}\left(  k\right)  =(1-\varkappa)\varkappa^{k},\qquad
\text{where}\qquad\varkappa=\frac{|J|\varepsilon}{\gamma_{-}}.
\]
We have
\[
\nu^{\prime}(0)=1-\varkappa,\quad\sum_{k\geq1}k\nu^{\prime}(k)=\frac
{\varkappa}{(1-\varkappa)},\quad\sum_{k>1}k\nu^{\prime}(k)=\frac{\varkappa
^{2}}{(1-\varkappa)}.
\]

Hence, from the domination argument we have

\begin{lemma}
\label{L3} Let the initial state $\nu_{t=0}$ has no customers. For
$\varkappa=\frac{|J|\varepsilon}{\gamma_{-}},$
\begin{equation}
\nu_{t}(0)\geq1-\varkappa,\quad t\geq0, \label{61}%
\end{equation}
so the fraction of non-empty queues is $O(\varepsilon)$;
\[
\sum_{k\geq1}k\nu_{t}(k)\leq\frac{\varkappa}{(1-\varkappa)},
\]
so the expected length of the queue is $O(\varepsilon)$;
\begin{equation}
\sum_{k>1}k\nu_{t}(k)\leq\frac{\varkappa^{2}}{(1-\varkappa)}, \label{63}%
\end{equation}
so the expected length of non-single queues is $O(\varepsilon^{2})$.
\end{lemma}

Let $h(j)$ be the expected remaining number of services for a particle of
class $j$. Clearly, $h(j)$ depends only on the rooting matrix $P$ and is
finite since $\rho(P)<1$ (the network is open). Let us put
\[
h=\max_{j\in J}h(j).
\]
For a queue $x\in X$ we denote
\[
L(x)=\sum_{i=1}^{|x|}h(x_{i}),
\]
where $x=(j_{1},\dots,j_{|x|})$. Finally,
\[
L(\mu)=\int_{X}L(x)d\mu(x).
\]

As a consequence of Lemma \ref{L3}, we get the following upper bound on
$L(\nu_{t})$:

\begin{lemma}
\label{L3a} Under the hypothesis of Lemma \ref{L3}, there exists a constant
$C_{1}>0,$ which depends only on the routing matrix $P,$ such that
\begin{equation}
L(\nu_{t})\leq h\alpha(\nu_{t})+C_{1}\varepsilon^{2}, \label{71}%
\end{equation}
where $\alpha(\nu_{t})=1-\nu_{t}(\varnothing)$.
\end{lemma}

\begin{proof}
Note that $L(x)\leq h|x|$. We estimate $L(\nu_{t})$ separately over queues of
length one and over longer queues. The first contribution is $\leq h\alpha
(\nu_{t}),$ since $\nu_{t}(1)\leq\alpha(\nu_{t}),$ and for the second one we
use $\left(  \ref{63}\right)  .$
\end{proof}

\subsection{Bounds for the NLMP}

Now we return to the NLMP $\mathcal{H}$. Our goal is to prove that if
$\varepsilon>0$ is small, then any solution $\mu_{t}$ converges weakly to the
unique stationary solution $\mu^{\ast}$. We will do it, first, for the class
of \emph{finite weight} initial measures $\mu_{0}$, that is, for $\mu_{0}$
with a finite expected length of the queue:
\[
\sum_{x\in X}|x|\mu_{0}\left(  x\right)  <\infty.
\]
Note that the state $\mu_{t}$ is finite weight for all $t>0$ iff $\mu_{0}$ is
finite weight.

First, we will show that there exists a time moment $T\left(  \mu_{0}\right)
,$ such that the expected value of non-empty queues
\[
\alpha(\mu_{t})=1-\mu_{t}(\varnothing)
\]
in our system drops down below the value $C\varepsilon$ for some constant
$C>0$. Second, we prove that for some $\bar{C}>0$ it remains below the value
$\bar{C}\varepsilon$ at all later times. That implies that all the flow rates
in the system will stay below $D\varepsilon$ for some $D>0$.

To make the first step we will construct a Lyapunov function. To implement the
second step we `split' our process $\mathcal{H}$ after the time moment
$T\left(  \mu_{0}\right)  $ in two parts, such that the `smaller' one has all
the non-empty servers at the time $T\left(  \mu_{0}\right)  ,$ and the
`larger' part starts from the empty state at $T\left(  \mu_{0}\right)  .$

We begin by explaining the splitting construction. The splitting of the
process $\mathcal{H}=\left\{  \mu_{t},t\geq0\right\}  $ is a pair of processes
$\mu_{t}^{1}$ and $\mu_{t}^{2}$ on $X$ such that
\begin{equation}
\rho\mu_{t}^{1}+(1-\rho)\mu_{t}^{2}=\mu_{t},\quad t\geq0. \label{zzz}%
\end{equation}
In words, we mark some part of queues at $t=0$ and then consider the joint
evolution of marked and unmarked queues within the NLMP $\mathcal{H}$. The
processes $\mu_{t}^{1}$ and $\mu_{t}^{2}$ are also NLMP-s, but with
non-homogeneous external flows.

Formally, let $\mu_{0}=\rho\mu_{0}^{1}+(1-\rho)\mu_{0}^{2}$, where $\mu
_{0}^{i}$, $i=1,2$, are probability measures on $X$ and $0\leq\rho\leq1$. Let
us describe the dynamics of $\mu_{t}^{i}$, $i=1,2$, $t\geq0$ by means of
internal flows. We define, in the same manner as above,
\begin{equation}
u_{j}^{k}(t)=\sum_{x\in X}\mu_{t}^{k}(x)\sum_{r:x_{r}=j}\gamma(x,x_{r}),\qquad
j\in J,\qquad k=1,2, \label{YE2aa}%
\end{equation}
and%

\begin{equation}
w_{j}^{k}(t)=\sum_{i\in J}p_{ij}u_{i}^{k}(t),\qquad j\in J,\qquad k=1,2.
\label{YE3aa}%
\end{equation}
The total inflow will be defined in a different way:
\begin{equation}
v_{j}^{\prime}(t)=\lambda_{j}(t)+\rho w_{j}^{1}(t)+(1-\rho)w_{j}^{2}(t),\qquad
j\in J. \label{YE3ab}%
\end{equation}
Finally, the joint dynamics of $\mu_{t}^{k}$, $k=1,2$, satisfies the following
differential equations:
\begin{equation}
\dot{\mu}_{t}^{k}=\sum_{x\in X}\mu_{t}^{k}(x)\left[  \sum_{r=1}^{|x|}%
\gamma(x,x_{r})(\delta_{x\ominus x_{r}}-\delta_{x})+\sum_{j\in J}v_{j}%
^{\prime}(t)(\delta_{x\oplus j}-\delta_{x})\right]  , \label{YE7aa}%
\end{equation}
for $k=1,2$.

Summing up equations (\ref{YE7aa}) over $k=1,2$ with factors $\rho$ and
$(1-\rho)$, respectively, we conclude that (\ref{zzz}) holds. It is easy to
see that the evolution of each measure $\mu_{t}^{k}$ satisfies the equations
of an NLMP with a modified routing matrices: $p_{ij}^{1}=\rho p_{ij},$
$p_{ij}^{2}=\left(  1-\rho\right)  p_{ij}$ (fractions of outflows go to the
other half of the system) and an additional time-dependent Poisson inflow (the
customers that come from the other half of the system).

Now we are prepared to prove the main statement of this section.

\begin{lemma}
\label{L4} There exists a constant $\bar{C}>0$ and, for each initial
finite-weight measure $\mu_{0}$ on $X$, a time $T=T(\mu_{0})$ such that
\[
\alpha(\mu_{t})<\bar{C}\varepsilon\quad\text{for all}\quad t>T.
\]

\end{lemma}

\begin{proof}
We use the mean number of remaining service events $L(\mu)$ as the Lyapunov
function. The function $L(\mu)\ $is finite for any finite-weight measure $\mu$
on $X,$ since $L(x)\leq h|x|$ for any $x\in X$. Let us check that
\begin{equation}
\frac{d}{dt}L(\mu_{t})=\sum_{j\in J}\lambda_{j}(t)h(j)-S(\mu_{t}), \label{41}%
\end{equation}
where
\[
S(\mu)=\sum_{x\in X:x\neq\varnothing}\gamma(x)\mu\left(  x\right)
\]
is the mean service rate for a queue with the distribution $\mu$ (for empty
queues this rate is zero). Indeed, the arrival event of type $j$ adds the
value $h(j)$ to $L(\mu).$ On the other hand, the expected number of services
of a given customer decreases by $1$ after its service is over, which explains
the second term. Clearly, $S(\mu_{t})\geq{\gamma}_{-}\alpha(\mu_{t})$, where
${\gamma}_{-}$ is the minimal service rate (for nonempty queues). So from
$\left(  \ref{41}\right)  $ we see that if ${\gamma}_{-}\alpha(\mu
_{t})>|J|\varepsilon h,$ then the function $L(\mu_{t})$ decays. Hence at a
certain moment $t=T$ we will have $\alpha(\mu_{T})\leq\frac{|J|\varepsilon
h}{\gamma_{-}}$.

It remains to prove that for all $t\geq T$ the value of $\alpha(\mu_{t})$ will
never exceed $\bar{C}\varepsilon$ for some constant $\bar{C}$ which does not
depend on $\mu_{0}$ and $T$. To do this we use the splitting construction
expalined above, with $\rho=\alpha(\mu_{T}),$ the initial states $\mu_{T}%
^{1}\left(  \cdot\right)  =$ $\mu_{T}\left(  \cdot{\LARGE |}x\neq
\varnothing\right)  ,\ \mu_{T}^{2}\left(  \varnothing\right)  =$ $1.$ By
definition, the first portion -- $\rho\mu_{t}^{1}\left(  \cdot\right)  $ -- of
our process $\mu_{t}$ remains small for all times $t>T,$ since $\rho$ is
small. The total mass of the remaining measure $\left(  1-\rho\right)  \mu
_{t}^{2}$ is close to one, but the state $\mu_{T}^{2}$ is empty, and the rate
of external flow is of the order of $\varepsilon.$ Therefore one can expect
that the state $\mu_{t}^{2}$ is sparcely populated at all times $t>T.$ This is
indeed the case:

\textbf{Claim. }\textit{Let }$\lambda(t)=\sum_{j\in J}\lambda_{j}(t)$\textit{
and for all }$t\geq0$\textit{ we have }$\lambda(t)\leq\varepsilon$\textit{.
Let the initial state }$\mu_{0}\left(  \equiv\mu_{T}^{2}\right)  $\textit{ has
no customers, i.e. }$\mu_{0}\left(  \varnothing\right)  =1.$\textit{ Then
there exists a constant }$C>0$\textit{ such that for all}$\quad t\geq0$%
\[
\alpha(\mu_{t})\leq C\varepsilon.
\]
\textbf{ Proof of the Claim. }Note first that for some $C_{3}>0$ we have
\begin{equation}
\dot{L}(t)\leq C_{3}\varepsilon-\gamma_{-}\alpha(\mu_{t}),\quad t\geq0,
\label{YE4}%
\end{equation}
as it follows from $\left(  \ref{41}\right)  .$ Note also, that if $\dot
{L}(\mu_{t})\geq0,$ then $\alpha(\mu_{t})\leq\frac{C_{3}\varepsilon}%
{\gamma_{-}}\equiv\alpha_{+}.$ We will show that for all $t\geq0$ we have
$\alpha(\mu_{t})<\bar{\alpha}\equiv\left(  1+h\right)  \alpha_{+},$ provided
$\varepsilon$ is small.

Suppose the opposite, and let $\bar{T}>0$ be the first occasion when
$\alpha(\mu_{t})$ reaches the level $\bar{\alpha}$. Denote by $T_{+}$ the last
occasion before $\bar{T}$ when $\alpha(\mu_{T_{+}})=\alpha_{+}.$ Note that the
total inflow rate $|u(t)|,$ $t\leq\bar{T}$ satisfies the inequality
\[
|u(t)|\leq\varepsilon+\gamma_{+}\bar{\alpha}=C_{4}\varepsilon.
\]
Therefore, from $\left(  \ref{71}\right)  $ we have the inequality
\[
L(\mu_{T_{+}})\leq h\alpha_{+}+C_{5}\varepsilon^{2}.
\]
Since $\dot{L}(\mu_{t})\leq0$ for all $t\in\left[  T_{+},\bar{T}\right]  $, we
have
\begin{equation}
L(\mu_{\bar{T}})\leq L(\mu_{T_{+}})\leq h\alpha_{+}+C_{5}\varepsilon^{2}.
\label{Y5}%
\end{equation}
On the other hand, we always have that $L(\mu)\geq\alpha(\mu),$ so in
particular
\[
L(\mu_{\bar{T}})\geq\alpha(\mu_{\bar{T}})=h\alpha_{+}+\alpha_{+}%
\]
which contradicts (\ref{Y5}) when $\varepsilon>0$ is small enough, and the
claim follows.

The proof of Lemma \ref{L4} is thus also finished.
\end{proof}

\subsection{Stationary states of the NLMP}

Let us assume that the external inflow is stationary, that is, $\lambda
(t)\equiv\lambda=\left(  \lambda_{j}\right)  $, $t\geq0$. If the corresponding
stationary state $\mu^{\lambda}$ of $\mathcal{H}$ exists, then the constant
internal inflow rate $w^{\ast}$ is found from the equation
\begin{equation}
P\left(  \lambda+w^{\ast}\right)  =w^{\ast}, \label{QQ0}%
\end{equation}
that is,
\begin{equation}
w^{\ast}=\left(  I-P\right)  ^{-1}P\lambda. \label{QQ1}%
\end{equation}
Let us assume that $\lambda_{j}\leq\varepsilon$, $j\in J$, where
$\varepsilon>0$ is a small parameter.

\begin{lemma}
If $\varepsilon$ is small enough, then $\mathcal{H}$ has a unique stationary
state $\mu(t)\equiv\mu^{\ast}$.
\end{lemma}

The uniqueness follows from (\ref{QQ0}) and (\ref{QQ1}). For the existence, it
suffices to prove that the network without feedback is ergodic under the
stationary Poisson inflow of rate $\lambda+w^{\ast}$. The components of this
rate vector are of order $\varepsilon$, hence, the assertion follows from the
Foster criterion, see, for instance, \cite{B}, or directly from Lemma \ref{L3}.

\section{Ergodicity of the NLMP with low inflow}

\subsection{The main Theorem}

Now we turn to the main result of the paper.

\begin{theorem}
\label{T1} For a given $G$, there exists an $\varepsilon>0$ such that for all
$\lambda=\left(  \lambda_{j}\right)  $ with $\lambda_{j}<\varepsilon$, $j\in
J$ the NLMP dynamics $\mu_{t}$ with finite-weight initial measure $\mu_{0}$
converges to the unique equilibrium $\mu^{\lambda}$: for every $x\in X$
\[
\lim_{t\rightarrow\infty}|\mu_{t}(x)-\mu^{\lambda}\left(  x\right)  |=0.
\]

\end{theorem}

We begin the proof with a technical lemma.

\begin{lemma}
\label{L5} There exist $C,\varepsilon_{0}>0$ such that, for any $\varepsilon
<\varepsilon_{0}$ and any initial finite-weight measure $\mu_{0}$, there
exists a time $T=T(C,\varepsilon,\mu_{0})$ such that
\[
{\mathbb{E}}_{\mu_{t}}\left\vert x\right\vert <C\varepsilon\qquad\text{for
all}\quad t>T.
\]

\end{lemma}

\begin{proof}
This result follows from Lemma \ref{L3a} and Lemma \ref{L4}.
\end{proof}

\subsection{\label{cou1} The coupling of Non-Linear Markov Processes}

We are going to construct a coupling $\mathcal{H}^{\prime}\otimes
\mathcal{H}^{\prime\prime}$ of two NLMP $\mathcal{H}^{\prime}$ and
$\mathcal{H}^{\prime\prime},$ which live on the same graph, have the same
inflows, and differ only by their initial states, $\mu_{0}^{\prime}$ and
$\mu_{0}^{\prime\prime}$. We remind the reader that the coupling of the
processes $\mu_{t}^{\prime}$ and $\mu_{t}^{\prime\prime}$ is a process $M_{t}$
on $X\times X$ with marginals $\mu_{t}^{\prime}$ and $\mu_{t}^{\prime\prime}.$
In our case the process $M_{t}$ will be again a NLMP. Its initial state can be
any coupling of $\mu_{0}^{\prime}$ and $\mu_{0}^{\prime\prime};$ for example
the product $\mu_{0}^{\prime}\times\mu_{0}^{\prime\prime}$ will go.

The coupling we construct has a form of a sum: $M_{t}=R_{t}+W_{t};$ the two
(positive) processes $R_{t}$ and $W_{t}$ are interacting NLMP-s, except that
the masses $r_{t}\equiv R_{t}\left(  X\times X\right)  ,w_{t}\equiv
W_{t}\left(  X\times X\right)  $ of each of them can be less than one, while
of course $r_{t}+w_{t}=1.$ The processes $R_{t}$ and $W_{t}$ are processes on
pairs of queues, and we will call them the \textit{red }and \textit{white
}pairs; the customers in these queues will be likewise called \textit{red }and
\textit{white }customers. The choice of the initial states $M_{0}^{r}$ and
$M_{0}^{w}$ of the processes $R_{t}$ and $W_{t}$ are not very important; one
option can be $R_{0}=M_{0},W_{0}=0.$ The key property for us will be that the
process $W_{t}$ has its support on the diagonal $X\subset X\times X,$ so we
will be done if we will show that $w_{t}\rightarrow1$ as $t\rightarrow\infty.$

We start with the verbal\textbf{ Definition} of the evolution of the
(interacting) processes $R_{t}$ and $W_{t}.$ We then write down the
differential equation, equivalent to that definition.

As usual, we couple the (equal) inflows to $\mathcal{H}^{\prime}$ and
$\mathcal{H}^{\prime\prime}$ in such a way that the customers arrive in pairs
(of identical customers) simulataneously to $\mathcal{H}^{\prime}$ and
$\mathcal{H}^{\prime\prime}.$ The arriving external pairs of customers have
white colour. The fraction $w_{t}$ of arriving identical pairs goes into the
process $W_{t}.$ The remaining amount, $r_{t},$ goes into red process, $R_{t}%
$, and the corresponding white arriving customers become red instantly.

The service of white customers also happens in pairs, so the two paired
(identical) customers start and end their service sinchronously. After the
service is over, the newly created pair either leaves the network, or else go
to the \textit{same} node, according to the routing matrix, i.e. the routing
of the two customers is identical. The rule of staying white pair of customers
or becoming red customers are then the same as for the arriving external white pairs.

Each red customer $c$ is served individually. If $c$ was participating in the
process $\mathcal{H}^{\prime}$ (resp., $\mathcal{H}^{\prime\prime}$), then
after the service he either leaves the network or jumps to the another node,
still being in the process $\mathcal{H}^{\prime}$ (resp., $\mathcal{H}%
^{\prime\prime}$). If $c$ was belonging to the red pair of queues $\left(
q^{\prime}\oplus c,q^{\prime\prime}\right)  ,$ and if $q^{\prime}%
\neq\varnothing$ or $q^{\prime\prime}\neq\varnothing$, then the resulting pair
$\left(  q^{\prime},q^{\prime\prime}\right)  $ stays red, while in the case
when both $q^{\prime}=\varnothing$ and $q^{\prime\prime}=\varnothing,$ the red
pair $\left(  c,\varnothing\right)  $ turns into the pair $\left(
\varnothing,\varnothing\right)  ,$ which is declared white. If $c$ stays in
the network and enters into the pair $\left(  q^{\prime},q^{\prime\prime
}\right)  $ of queues, then the resulting pair $\left(  q^{\prime}\oplus
c,q^{\prime\prime}\right)  $ is declared red, independently of the colour of
the pair $\left(  q^{\prime},q^{\prime\prime}\right)  .$

Summarizing, we see that a white particle turns red in two cases: if it joins
a red queue, or if a red particle arrives to a white pair of queues (then all
the white customers in both queues turn red). The queues in the white pair are
identical, while the queues in the red pair may be identical or different. The
pair $\left(  \varnothing,\varnothing\right)  $ is always white.

Our goal will be to prove that the fraction of red pairs $r_{t}$ vanishes as
$t\rightarrow\infty.$ It will imply that the limits of $\mu_{t}^{\prime}$ and
$\mu_{t}^{\prime\prime}$, as $t\rightarrow\infty,$ exist and coinside with
$\mu^{\lambda}.$ In the end we will drop the assumption of $\mu_{0}^{\prime}$
to be finite-weight.

\subsection{\label{cou2} Equations for the coupled NLMP}

As was explained above, the state of our coupled process is given by the two
measures -- white and red:
\[
\left\{  W_{t}\left(  x\right)  ,R_{t}\left(  y,z\right)  :\quad x,y,z\in
X\right\}  ,
\]
where we think of $x$ as a pair $\left(  x,x\right)  \in X\times X$ on the
diagonal. Sometime we will omit the subscript $t.$

Given such a state, let us define the corresponding -- white and red -- flows.
The outflows are
\[
u_{j}(W)=\sum_{x\in X}W(x)\sum_{i:x_{i}=j}\gamma(x,i),
\]%
\[
u_{j}^{\prime}(R)=\sum_{y,z}R(y,z)\sum_{i:y_{i}=j}\gamma(y,i),
\]%
\[
u_{j}^{\prime\prime}(R)=\sum_{y,z}R(y,z)\sum_{i:z_{i}=j}\gamma(z,i).
\]
The corresponding internal inflows are
\[
v_{j}(W)=\sum_{k}p_{kj}u_{k}(W),
\]%
\[
v_{j}^{\prime}(R)=\sum_{k}p_{k,j}u_{j}^{\prime}(R),
\]%
\[
v_{j}^{\prime\prime}(R)=\sum_{k}p_{k,j}u_{j}^{\prime\prime}(R).
\]
Let us define $v_{j}\left(  W,R\right)  =v_{j}(W)+v_{j}^{\prime}%
(R)+v_{j}^{\prime\prime}(R).$

Now we write the differential equations for the white and red measures $W$ and
$R,$ which describe formally the evolution explained in the previous section.
\begin{align*}
\frac{d}{dt}W_{t}(x)  &  =-W_{t}(x)\left[  \sum_{j}\left(  \lambda_{j}%
+v_{j}(W_{t},R_{t})\right)  +\sum_{i=1}^{|x|}\gamma(x,i)\right] \\
+\ W_{t}(\bar{x}  &  :\bar{x}\oplus j=x)(\lambda_{j}+v_{j}(W_{t}))+\sum
_{i=1}^{|x|+1}\sum_{\tilde{x}:\tilde{x}\ominus\tilde{x}_{i}=x}W_{t}(\tilde
{x})\gamma(\tilde{x},\tilde{x}_{i})\\
&  +\ \mathbb{I}_{\left\{  x=\varnothing\right\}  }\left(  \sum_{j}\left[
R_{t}(j,\varnothing)+R_{t}(\varnothing,j)\right]  \gamma(j)\right)  ;\
\end{align*}
\begin{align*}
\frac{d}{dt}R_{t}(y,z)  &  =-R_{t}(y,z)\left[  \sum_{j}\left(  \lambda
_{j}+v_{j}(W_{t},R_{t})\right)  +\sum_{i=1}^{|y|}\gamma(y,i)+\sum_{i=1}%
^{|z|}\gamma(z,i)\right] \\
+R_{t}(\bar{y}  &  :\bar{y}\oplus j=y,\bar{z}:\bar{z}\oplus j=z)(\lambda
_{j}+v_{j}(W_{t}))\\
+R_{t}(\bar{y}  &  :\bar{y}\oplus j=y,z)v_{j}^{\prime}(R)\ +\ R_{t}(y,\bar
{z}:\bar{z}\oplus j=z)v_{j}^{\prime\prime}(R)\\
&  +\mathbb{I}_{\left\{  (y,z)=\left(  j,\varnothing\right)  \right\}  }%
W_{t}(\varnothing)v_{j}^{\prime}(R)\ +\mathbb{I}_{\left\{  (y,z)=\left(
\varnothing,j\right)  \right\}  }W_{t}(\varnothing)v_{j}^{\prime\prime}(R)\\
&  +\mathbb{I}_{\left\{  (y,z)\neq\left(  \varnothing,\varnothing\right)
\right\}  }\left[  \sum_{i=1}^{|y|+1}\sum_{\bar{y}:\bar{y}\ominus\bar{y}%
_{i}=y}R(\bar{y},z)\gamma(\bar{y},\bar{y}_{i})+\sum_{i=1}^{|z|+1}\sum_{\bar
{z}:\bar{z}\ominus\bar{z}_{i}=z}R(y,\bar{z})\gamma(\bar{z},\bar{z}%
_{i})\right]  .
\end{align*}

\subsection{Proof of stability}

Now the colored dynamics is explicitely defined; it is a non-linear Markov
process on the space of colored pairs of queues. Any pair is either white (and
then the queues are identical and synchronized) or red (and then the queues
evolve independently). The theorem \ref{T1} follows from the following

\begin{theorem}
If $\varepsilon$ is small enough, then $r_{t}\rightarrow0$ as $t\rightarrow
\infty.$
\end{theorem}

\begin{proof}
First, we start the two NLMP $\mathcal{H}^{\prime}$ and $\mathcal{H}%
^{\prime\prime}$ and run them for the time $T$, given by Lemma \ref{L4}. After
that time we couple them as was explained above. The rates of all the flows in
the coupled system do not exceed $C_{1}\varepsilon$ at any time $t\geq T$.
Denote by $h(t)$ the expected number of remaining services \textit{of red
particles only} in both $\mathcal{H}^{\prime}$ and $\mathcal{H}^{\prime\prime
}$. By assumption, $h(0)$ is finite and so $h(t)$ remains finite for all
$t\geq0$. The change of $r_{t}$ and $h(t)$ in time is due to two factors. The
first one is the termination of the service of red particles; it contributes
to decline of $h(t)$, which happens with the rate $Cr_{t}.$ The second one is
the creation of new red particles, and it makes $h\left(  t\right)  $ to
increase, either when

\begin{enumerate}
\item white particles arrive to red queues, or when

\item red particles arrive to white queues.
\end{enumerate}

In the case 1 the rate of emergence of the \textit{new} red particles (and
hence the rate of growth of $h(t)$) is at most of the order of $\varepsilon
r_{t}$, since the flow of red particles has order $r_{t}$, the fraction of
\textit{non-empty} white queues has order at most $\varepsilon$ and the mean
length of these queues is of order $1$. In the case 2, the rate of growth of
$h(t)$ is also of order $\varepsilon r_{t}$ since the flow of white particles
has order $\varepsilon$ and the fraction of red queues has order $r_{t}$.
Summarizing, we arrive at the following differential inequality:
\[
\frac{d}{dt}h(t)\leq-\left(  C-2\varepsilon\right)  r_{t}.
\]
Since $h(t)$ is positive, the integral $\int_{0}^{\infty}r_{t}dt\ $is finite.
Together with the fact that $r_{t}\geq0$ and the derivative of $r_{t}$ is
bounded, it implies that
\[
\lim_{t\rightarrow\infty}r_{t}=0.
\]

\end{proof}

\section{Infinite expected initial length}

Here we say few words on extension of Theorem \ref{T1} to the case of infinite
expected initial queue length. We give a sketch of the proof.

In order to reduce the situation to that of a finite initial length, we
separate a small fraction $\beta>0$ of queues so that the remaining ones are
of finite expected length at $t=0$. The value of $\beta$ can be chosen
arbitrarily small.

The flows to and from the separated part are of order $\beta$. We need to
prove that, for finite initial weights and for different inflows that differ
by order $\beta$, the solutions approach each other up to a distance of order
$\beta$ as well.

First, we prove that the fraction of red queues must eventually drop down to
the order $\beta$. This follows from the differential equations for the
dynamics of $L(t)$ and $r_{t}$. So we can suppose that $r_{T}\sim\beta$ at
some time $T.$

At this time $T$ we separate, additionally, all the red queues (their fraction
is $r_{T}$) and consider the remaining process that starts with no red queues.
Then we can use the argument similar to that in the proof of Lemma \ref{L5}
and show that $r_{t}$ remains of order $\beta$ for all $t\geq T$. The
convergence to the equilibrium follows since $\beta$ can be chosen arbitrarily small.


\end{document}